\documentclass[12pt]{article}

\usepackage{amsmath,amssymb,amsthm,amsfonts}
\usepackage[top=2.5cm,bottom=2.5cm,left=2.5cm,right=2.5cm]{geometry}
 \usepackage{tikz,color}
\usepackage{graphicx,ifpdf,cite}
\usepackage{float}

\vspace{8mm}

\title{Ordering connected graphs by their Kirchhoff indices \thanks{
 The first author is supported by NNSF of China (No. 11201227), China Postdoctoral Science Foundation (2013M530253) and Natural Science Foundation of Jiangsu Province (BK20131357), the second author was supported by National Research Foundation funded by the Korean government with the grant No. 2013R1A1A2009341, the third author is supported by NNSF of China (No. 11271256). \newline Email addresses: kexxu1221@126.com(K. Xu), kinkardas2003@gmail.com(K. C. Das), \newline xiaodong@sjtu.edu.cn(X.D. Zhang).}
 }
 \author{Kexiang Xu$\/^{a},$ Kinkar Ch. Das$\/^{b},$ Xiao-Dong Zhang$\/^{c}$\\\\
 $^{a}$ \small  College of Science, Nanjing University of
 Aeronautics \& Astronautics,\\
 \small Nanjing, Jiangsu 210016, P.R. China\\
 $^{b}$ \small Department of Mathematics, Sungkyunkwan University,\\
 \small Suwon 440-746, Republic of Korea\\
 $^{c}$ \small Department of Mathematics, Shanghai Jiao Tong University,\\
\small 800 Dongchuan road, Shanghai, 200240, P.R. China}
 \date{}

\begin{document}

\newtheorem{corollary}{Corollary}[section]
\newtheorem{theorem}{Theorem}[section]
\newtheorem{remark}{Remark}[section]
\newtheorem{lemma}{Lemma}[section]
\newtheorem{conjecture}{Conjecture}[section]
\newtheorem{problem}{Problem}[section]

\renewcommand{\proofname}{\textup{\textbf{Proof}}}

 \maketitle
\begin{center}
 (Received  Nov. 24, 2014)
 \end{center}
\baselineskip=0.23in

\begin{abstract}
 The Kirchhoff index
$Kf(G)$ of a graph $G$ is the sum of resistance distances between
all unordered pairs of vertices, which was introduced by Klein and Randi\'c. In this paper we characterized all extremal graphs with Kirchhoff index among all graphs obtained by deleting $p$ edges from a complete graph $K_n$ with $p\leq\lfloor\frac{n}{2}\rfloor$ and obtained a sharp upper bound on the Kirchhoff index of these graphs. In addition, all the graphs with the first to ninth maximal Kirchhoff indices are completely determined among all connected graphs of order $n>27$.
\end{abstract}

 \textit{Keywords:} Graph; Distance (in graph);
Kirchhoff index; Laplacian spectrum\\

\textit{AMS Subject Classifications {\em(2010)}: }~05C50, 05C12,
05C35.


 \baselineskip=0.30in

\section{Introduction}

\ \indent Let $G$ be a connected graph with vertices labeled as  $v_1,\,v_2,
\ldots,\,v_n$. The distance between vertices $v_i$ and $v_j$,
denoted by $d_G(v_i,\,v_j$), is the length of a shortest path
between them. The famous Wiener index $W(G)$ \cite{Wiener} is the
sum of distances between all unordered pairs of vertices, that is,
 $W(G) =\sum
_{i<j}d_G(v_i,\,v_j).$

In 1993, Klein and Randi\'c \cite{Klein93} introduced a new distance
function named resistance distance based on electrical network
theory. They viewed  $G$ as an electrical network $N$ by replacing
each edge of $G$ with a unit resistor, the resistance distance
between $v_i$ and $v_j$, denoted by $r_G(v_i,\,v_j)$, is defined to
be the effective resistance between them in $N$. Similar to the long
recognized shortest path distance, the resistance distance is also
intrinsic to the graph, not only with some nice purely mathematical
and physical interpretations \cite{Klein93, Klein97}, but with a
substantial potential for chemical applications.

In fact, the shortest-path might be imagined to be more relevant
when there is corpuscular communication (along edges) between two
vertices, whereas the resistance distance might be imagined to be
more relevant when the communication is wave- or fluid-like. Then
the chemical communication in molecules is rather wavelike suggests
the utility of this concept in chemistry. So in recent years, the
resistance distance was well studied in mathematical and chemical
literatures \cite{BKlein02,Bapat,Bonchev94,DA1,DA2,DXG2012,DChen2013,DChen2014}.

Analogue to Wiener index, the Kirchhoff index (or resistance index)
\cite{Bonchev94} is defined as
$$Kf(G) =\sum _{i<j}r_G(v_i,\,v_j).$$

As a useful structure-descriptor,  the computation of Kirchhoff
index is a hard problem \cite{BKlein02}, but one may compute the
specific classes of graphs. Since for trees, the
Kirchhoff index and the Wiener index coincide. It is possible to
study the Kirchhoff index of topological structures containing
cycles. Throughout this paper we denote by $P_n$ (resp. $C_n$, $K_n$) denote the path graph (resp. cycle graph, complete graph) on $n$
vertices. Some nice mathematical results can be found in \cite{Lukovits99,Yang2008}.

All graphs considered in this paper are finite and simple. For two
nonadjacent vertices $v_i$ and $v_j$, we use $G+e$ to denote the
graph obtained by inserting a new edge $e=v_i\,v_j$ in $G$.
Similarly, for $e\in E(G)$ of graph $G$, let $G-e$ be the subgraph
of $G$ obtained by deleting the edge $e$ from $E(G)$. The complement
of graph $G$ is always denoted by $\overline{G}$. For two vertex
disjoint graphs $G_1$ and $G_2$, we denote by $G_1\bigcup G_2$ the
graph which consists of two connected components $G_1$ and $G_2$.
The \textit{join} of $G_1$ and $G_2$, denoted by $G_1 \bigvee G_2$,
is the graph with vertex set $V (G_1)\bigcup V (G_2)$ and edge set
$E(G_1)\bigcup E(G_2)\bigcup \{u_iv_j : u_i\in V(G_1), v_j\in
V(G_2)\}$. For other undefined notation and terminology from graph
theory, the readers are referred to \cite{BM1976}.

For a graph $G$ with vertex set $V=\{v_1,\,v_2,\ldots,\,v_n\}$, we
denote by $d_i$ the degree of the vertex $v_i$ in $G$ for $i=1,\,2,
\ldots , n$. Assume that $A(G)$ is the $(0, 1)$-adjacency matrix of
$G$ and $D(G)$ is the diagonal matrix of vertex degrees. The
Laplacian matrix of $G$ is $L(G)=D(G)-A(G)$. The Laplacian
polynomial $Q(G, \lambda)$ of $G$ is the characteristic polynomial
of its Laplacian matrix, $Q(G,\lambda)=det(\lambda I_n-L(G))=
\sum\limits_{k=0}^{n}(-1)^kc_k\lambda^{n-k}$. The Laplacian matrix
$L(G)$ has nonnegative eigenvalues
$n\geq\mu_1\geq\mu_2\geq\cdots\geq\mu_n=0$ \cite{CMS1995}. Denote by
$S(G)=\{\mu_1,\,\mu_2,\ldots,\,\mu_n\}$ the spectrum of $L(G)$,
i.e., the Laplacian spectrum of $G$. If the eigenvalue $\mu_i$
appears $l_i>1$ times in $S(G)$, we write them as $\mu_i^{(l_i)}$
for the sake of convenience.

  In 1996, Gutman and Mohar \cite{Gut1996} obtained the following nice result, by which a relation is established between Kirchhoff index and Laplacian spectrum:
  \begin{equation}
 Kf(G)=n\sum\limits_{i=1}^{n-1}\frac{1}{\mu_i}~\label{e1}
 \end{equation}
for any connected graphs of order $n\geq 2$.

Let ${\cal{G}}(n)$ be the set of connected graphs of order $n$. In
this paper, we determined the first to ninth minimal Kirchhoff
indices of graphs from ${\cal{G}}(n)$ with $n>9$; also characterized
all the graphs from ${\cal{G}}(n)$ with $n>27$ with the first to
ninth maximal Kirchhoff indices.

\section{Preliminaries}
 \ \indent In this section we will list some known lemmas as necessary preliminaries.
 \begin{lemma}\label{NEW0} {\em(\cite{GMS1990})} Let $G$ be a graph and $G'=G+e$ the graph obtained by inserting a new edge into $G$. Then we have
$$\mu_1(G')\geq \mu_1(G)\geq \mu_2(G')\geq\mu_2(G)\geq\cdots\geq\mu_n(G')=\mu_n(G)=0.$$
\end{lemma}

Combining Lemma \ref{NEW0} and the fact that
$\sum\limits_{i=1}^{n-1}
\mu_i(G+e)-\sum\limits_{i=1}^{n-1}
\mu_i(G) = 2$, by the equation (\ref{e1}),
the following lemma can be easily obtained.
\begin{lemma}\label{NEW1} \em{(\cite{Lukovits99})} Let $G$ be a connected graph with $e\in E(G)$ and two nonadjacent
vertices $v_i$ and $v_j$ in $V(G)$. Then we have
\begin{itemize}

 \item [$(1)$]  $Kf(G-e)> Kf(G)$ where $G-e$ is connected;

 \item [$(2)$]  $Kf(G)>Kf(G+e')$ where $e'=v_iv_j$.
 \end{itemize}
\end{lemma}
Based on Lemma \ref{NEW1} $(1)$, the corollary below follows immediately.
\begin{corollary} \label{CO0} Suppose that $G$ is a connected graph of order $n$ and with $m\geq n$ edges and with $T$ as its spanning tree. Then we have $Kf(G)<Kf(T)$.
\end{corollary}

\begin{lemma}\label{NEW2} {\em(\cite{Merris1994})} Let $G$ be a graph of order $n$ with $S(G)=\{\mu_1,\,\mu_2, \ldots ,\,\mu_{n-1}, 0\}$.
Then $S(\overline{G})=\{n-\mu_1,\, n-\mu_2,\ldots ,\,n-\mu_{n-1},
0\}$.
\end{lemma}

\begin{lemma} \label{NEW4} {\em(\cite{Klein93})}
Let $G$ be a connected graph. Then we have $W(G) \geq Kf(G),$ with
equality if and only if $G$ is a tree.
\end{lemma}

Before listing this problem, we first introduce some necessary
notations and definitions. A vertex $v$ of a tree $T$ is called a
\textit{branching point} if $d(v) \geq 3$. A tree is said to be
starlike if exactly one of its vertices has degree greater than two.
Let $P_n$ denote the path on $n$ vertices. By
$T_n(n_1,\,n_2,\ldots,\,n_k)$ we denote the starlike tree which has
a vertex $v$ of degree $k\geq 3$ and which has the property
$$T_n(n_1,\,n_2,\ldots,\,n_k)-v=P_{n_1}\cup P_{n_2} \cup \cdots\cup P_{n_k}.$$
This tree has $n_1+n_2+\ldots+n_k+1=n$ vertices and assumed that
$n_1\geq n_2\geq \ldots\geq n_k\geq 1.$ We say that the starlike
tree $T_n(n_1,\,n_2,\ldots,\,n_k)$ has $k$ branches, the lengths of
which are $n_1,\,n_2,\ldots,\,n_k$, respectively.

\vspace*{3mm}

Note that any tree with exactly  one branching point is a starlike
tree. Assume that $T$ is a tree of order $n$ with exactly two
branching points $v_1$ and $v_2$ with $d(v_1) = r$ and $d(v_2) = t$.
The orders of $r-1$ components, which are paths, of $T-v_1$ are
$p_1,\ldots,\, p_{r-1}$, the order of the component which is not a
path of $T-v_1$ is $p_r=n-p_1-\cdots-p_{r-1}-1$. The orders of $t-1$
components, which are paths, of $T-v_2$ are $q_1,\ldots, q_{t-1}$,
the order of the component which is not a path of $T-v_2$ is
$q_t=n-q_1-\cdots-q_{t-1}-1$. We denote this tree by $T=T_n(
p_1,\ldots,\, p_{r-1}; q_1,\ldots, \,q_{t-1})$, where $r\leq t$,
$p_1\geq \cdots \geq p_{r-1}$ and $q_1 \geq \cdots \geq q_{t-1}$.
\begin{figure}[ht!]
\begin{center}
\begin{tikzpicture}[scale=1,style=thick]
\def\vr{3pt}

\draw (-5.8,0) -- (-5.0,0) -- (-4.2,0) -- (-3.4,0) -- (-2.6,0)-- (-1.8,0) -- (-1.0,0)-- (-0.2,0) -- (0.6,0);
\draw (-3.4,0) -- (-3.4,0.8);
\draw (1.8,0.5) -- (2.4,0) -- (1.8,-0.5);
\draw (2.4,0) -- (3.2,0) -- (4.0,0) -- (4.8,0) -- (5.6,0);
\draw (6.4,0.5) -- (5.6,0) -- (6.4,-0.5);

\draw (-3.4,0.8)  [fill=black] circle (\vr);
\draw (-5.0,0)  [fill=black] circle (\vr);
\draw (-5.8,0)  [fill=black] circle (\vr);
\draw (-4.2,0)  [fill=black] circle (\vr);
\draw (-3.4,0)  [fill=black] circle (\vr);
\draw (-2.6,0)  [fill=black] circle (\vr);
\draw (-1.8,0)  [fill=black] circle (\vr);
\draw (-1.0,0)  [fill=black] circle (\vr);
\draw (-0.2,0)  [fill=black] circle (\vr);
\draw (0.6,0)  [fill=black] circle (\vr);

\draw (1.8,0.5)  [fill=black] circle (\vr);
\draw (1.8,-0.5)  [fill=black] circle (\vr);
\draw (2.4,0)  [fill=black] circle (\vr);

\draw (3.2,0)  [fill=black] circle (\vr);
\draw (4.0,0)  [fill=black] circle (\vr);
\draw (4.8,0)  [fill=black] circle (\vr);
\draw (5.6,0)  [fill=black] circle (\vr);

\draw (6.4,-0.5) [fill=black] circle (\vr);
\draw (6.4,0.5) [fill=black] circle (\vr);
\draw (-2.8,-0.8) node {$T_{10}(5,3,1)$};
\draw (4.2,-0.8) node {$T_9(1^2;1^2)$};

\end{tikzpicture}
\end{center}
\caption{ The trees $T_{10}(5,3,1)$ and $T_9(1^2;1^2)$}
\label{G1}
\end{figure}
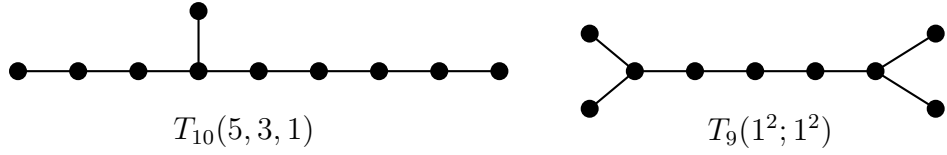

For convenience, when considering the trees
$T_n(n_1,\,n_2,\ldots,\,n_k,\ldots,\,n_m)$ or $T_n(
p_1,\ldots,\,p_k,\\\ldots,\,p_{r-1};$ $
q_1,\ldots,\,q_k,\ldots,\,q_{t-1})$, we use the symbols
$n_{k}^{l_{k}}$ or $p_{k}^{l_{k}}$ (resp. $q_{k}^{l_{k}}$) to
indicate that  the number of $n_{k}$ or $p_k$ (resp. $q_k$) is
$l_{k}>1$  in the following. For example,
$T_{15}(2,\,2,\,3,\,3,\,4)$ will be written
 as $T_{15}(2^{2},\,3^{2},\,4)$. As another two examples, the trees $T_{10}(5,\,3,\,1)$ and $T_9(1^2;\,1^2)$ are shown in Figure \ref{G1}.

 In the following lemma the partial result in \cite{LLL2010} are summarized.
\begin{lemma}\label{LEM5} {\em(\cite{LLL2010})} ~Suppose that $T$ is a tree of order  $n\geq 9$.
Then we have
\begin{eqnarray*}W(P_n)&>&W(T_n(n-3,\,1^2))>W(T_n(n-4,\,2,\,1))>W(T_n(1^2;\,1^2))>W(T_n(n-5,\,3,\,1))\\ &>&W(T_n(n-4,\,1^3))=W(T_n(1^2;\,2,\,1))>W(T_n(n-6,\,4,\,1))>W(T).\end{eqnarray*}\end{lemma}
Combining Lemmas \ref{NEW4} and \ref{LEM5}, the following corollary
can be easily obtained.

\begin{corollary} \label{CO1} Suppose that $T$ is a tree of order  $n\geq 9$.
Then we have
\begin{eqnarray*}Kf(P_n)&>&Kf(T_n(n-3,\,1^2))>Kf(T_n(n-4,\,2,\,1))>Kf(T_n(1^2;\,1^2))>Kf(T_n(n-5,\,3,\,1))\\ &>&Kf(T_n(n-4,\,1^3))=Kf(T_n(1^2;\,2,\,1))>Kf(T_n(n-6,\,4,\,1))>Kf(T).\end{eqnarray*}
\end{corollary}

Let $P_n^k$ be the graph obtained by identifying a pendent vertex of  a path of length $n-k+1$ with one vertex of a cycle $C_k$.

\begin{lemma}\label{LEM6} {\em(\cite{Yang2008})} For any connected graph $G$ of order $n>3$ and with $n>3$ edges, we have
\begin{eqnarray*}
 Kf(G)\leq \frac{n^3-11n+18}{6}
 \end{eqnarray*}
 with equality if and only if $G\cong P_n^3$.
 \end{lemma}

 Denote by $C_{p,q}^{l}$ the graph which is formed by two disjoint cycles $C_p$ and $C_q$ linked by a path of length $l$ (see Figure \ref{G2}). In  \cite{ZhangYang09}, the authors
 determined the graph which maximizes the Kirchhoff index among all connected graphs of order $n$ with $n$ edges and exactly two cycles. Recently, Feng, Yu et al. and one of the
 present authors \cite{FYXJ2014} completely characterized the extremal graph with maximal Kirchhoff index among all connected graphs of order $n$ and with $n+1$ edges.
 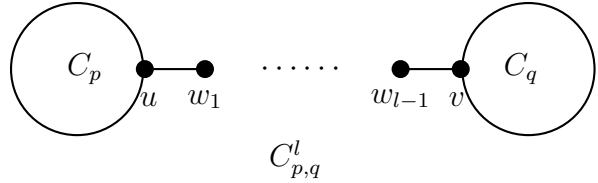
\begin{figure}[ht!]
\begin{center}
\begin{tikzpicture}[scale=1,style=thick]
\def\vr{3pt}
\draw (0,0) circle (25pt);
\draw (6,0) circle (25pt);
\draw (0.9,0) -- (1.7,0);
\draw (4.3,0) -- (5.1,0);

\draw (0.9,0)  [fill=black] circle (\vr);
\draw (5.1,0)  [fill=black] circle (\vr);
\draw (1.7,0)  [fill=black] circle (\vr);
\draw (4.3,0)  [fill=black] circle (\vr);
\draw (3.0,0) node {$\cdots\cdots$};

\draw (0.1,0) node {$C_p$};
\draw (5.9,0) node {$C_q$};
\draw (0.95,-0.4) node {$u$};
\draw (5.05,-0.4) node {$v$};

\draw (1.7,-0.4) node {$w_1$};
\draw (4.3,-0.4) node {$w_{l-1}$};

\draw (2.9,-1.2) node {$C_{p,q}^{l}$};

\end{tikzpicture}
\end{center}
\caption{ The graph $C_{p,q}^{l}$}
\label{G2}
\end{figure}

 \begin{lemma}\label{LEM7} {\em(\cite{FYXJ2014})} Let $G$ be a connected graph of order $n$ and with $n+1$ edges $(n\geq 8)$. Then we have
  \begin{eqnarray*}
 Kf(G)\leq \frac{n^3-21n+36}{6}
 \end{eqnarray*} with equality if and only if $G\cong C_{3,3}^{n-5}$.
 \end{lemma}

 An invariant related to Kirchhoff index is defined \cite{Yang2008} as follows: $Kf_{v_i}(G)=\sum\limits_{j\neq i}r_G(v_i,v_j)$. In the following lemma a nice formula is
 presented on Kirchhoff index of a graph with cut vertices.
 \begin{lemma}\label{LEM8}{\em(\cite{ZhangYang09})} Let $x$ be a cut vertex of connected graph $G$ such that $G=G_1\bigcup G_2$, $V(G_1)\bigcap V(G_2)=\{x\}$ and $|V(G_i)|=n_i$
 for $i=1,2$. Then we have \begin{eqnarray*}Kf(G)=Kf(G_1)+Kf(G_2)+(n_1-1)Kf_x(G_2)+(n_2-1)Kf_x(G_1).\end{eqnarray*}
 \end{lemma}
Note that (\cite{XLDGF2014}) $P_n$ has uniquely the largest Wiener index among all trees of order $n$. From Lemma \ref{LEM8}, the corollary below follows immediatey.
 \begin{corollary} \label{CAD1} Let $G_0$ be a connected graph with $v_0\in V(G_0)$ and $T_t$ a tree of order $t\geq 2$ with $x\in V(T_t)$. Assume that $G$ is a graph obtained by identifying the vertex $v_0$ in $G_0$ with $x\in T$ and $G^{\prime}$ is obtained by identifying $v_0 \in G_0$ with a pendent vertex of path $P_t$. Then \begin{eqnarray*}Kf(G)\leq Kf(G^{\prime})\end{eqnarray*} with equality holding if and only if $G\cong G^{\prime}$, i.e., $T_t\cong P_t$ with $x$ being a pendent in $T_t$. \end{corollary}
 \begin{lemma}\label{LEA9} {\em(\cite{Yang2008})}
 Among all connected graph of order $n$ with $n$ edges and cycle length $k$, the graph $P_n^k$ has uniquely the maximal Kirchhoff index. \end{lemma}
\section{Main results}
\ \indent In this section, we will order all the graphs from ${\cal{G}}(n)$ with $n$ being not very small by their Kirchhoff indices. In what follows, we will deal with the
two cases, respectively, for graphs from ${\cal{G}}(n)$ with smaller Kirchhoff indices and with larger Kirchhoff indices.

\subsection{The ordering of connected graphs with smaller Kirchhoff indices}
\ \indent Lukovits et al. \cite{Lukovits99}
showed that, among all connected graphs of order $n$,  $Kf(G) \geq n -1$ with equality if and only if $G$ is
complete graph $K_n$. In the following it suffices to order the graphs from ${\cal{G}}(n)\setminus \{K_n\}$ by their Kirchhoff indices.

For convenience, for a subgraph $G_0$ of $K_n$, we denote by $K_n-G_0$ the graph obtained by deleting all edges of $G_0$ from $K_n$. From the structure of $K_n-G_0$, we claim
that $\overline{K_n-G_0}\cong \overline{K_{n-|V(G_0)|}}\bigcup G_0$.  For the consistency of sign, we write $G_1(n)=K_n$ and $G_2(n)=K_n-K_2$. Moreover, let $G_3(n)=K_n-2K_2$
and $G_4(n)=K_n-K_{1,2}$. Next we consider the graphs obtained by deleting three edges from $K_n$.
Assume that
$$G_5(n)=K_n-3K_2;~~~
G_6(n)=K_n-(K_{1,2}\cup K_2);~~~
G_7(n)=K_n-P_4;$$
$$G_8(n)=K_n-C_3;~~~
G_9(n)=K_n-K_{1,3}. $$

In the following theorem the graphs from ${\cal{G}}(n)$ with $n\geq 11$ and with first to ninth minimal Kirchhoff indices are completely determined.
\begin{theorem}\label{TH1.0} {\em (\cite{DXG2012})} Let $n\geq 11$ and  $G\in {\cal{G}}(n)$ but other than any graph from the set $\{G_i(n)|i\in \{1,2,\ldots,9\}\}$. Then we have
\begin{eqnarray*}Kf(G)>Kf(G_9(n))>Kf(G_8(n))>Kf(G_7(n))>Kf(G_6(n))>Kf(G_5(n))\\ \ \indent\ \indent>Kf(G_4(n))>Kf(G_3(n))>Kf(G_2(n))>Kf(G_1(n)).\end{eqnarray*} \end{theorem}
In view of Theorem \ref{TH1.0}, naturally we will ask a related problem as follows:

\textit{For an integer $4\leq p\leq \Big\lfloor\displaystyle{\frac{n}{2}}\Big\rfloor$, which graph has the extremal Kirchhoff index among all connected graphs obtained by deleting $p$ edges from $K_n$?}

Before solving the above problem, we need a related lemma as follows:
 \begin{lemma} \em{(\cite{DA0})} \label{m2}
 Let $G$ be a connected graph with at least one edge. Then
 \begin{equation}                    \label{t10}
 \mu_1(G)\leq \max_{v_iv_j\in E(G)}|N_i\cup N_j|
 \end{equation}
 where $N_i$ is the neighbor set of vertex $v_i\in V(G)$\,. This upper bound for $\mu_1(G)$ does not exceed $n$.
 \end{lemma}
 In the following theorem we will give a complete solution of this problem for the minimal case.
 \begin{theorem} \label{NEWTH1} For any integer $2\leq p\leq \Big\lfloor\displaystyle{\frac{n}{2}}\Big\rfloor$ and any graph $G$ obtained by deleting $p$ edges from $K_n$, we have
 \begin{eqnarray}
  Kf(G)\geq n-1+\frac{2p}{n-2}\label{ps1}
 \end{eqnarray}
 with equality holding in {\em(\ref{ps1})} if and only if $G\cong K_n-p\,K_2$.
 \end{theorem}

 \begin{proof} Denote by $\overline{\mu}_i$ with $i=1,\,2,\ldots,\,n$ the non-increasing Laplacian eigenvalues of $\overline{G}$. By Lemma \ref{NEW2}, we have
 $\overline{\mu}_i=n-\mu_{n-i}$ for $i=1,\,2,\ldots,\,n-1$. Since $\overline{G}$ is the complement graph of $G$\,, we have $\overline{m}=p$ with
 $2\leq p\leq \left\lfloor\frac{n}{2}\right\rfloor$,  where $\overline{m}$ is the number of edges in $\overline{G}$. Since
            $$\overline{m}=p\leq \left\lfloor\frac{n}{2}\right\rfloor\,,$$
 $\overline{G}$ must be a disconnected graph. Let $k$ be the number of connected components in $\overline{G}$. Also let $\overline{n}_i$ and $\overline{m}_i$
 be the number of vertices and number of edges in the $i$-th component of $\overline{G}$ such that $\overline{n}_1\geq \overline{n}_2\geq \cdots\geq
 \overline{n}_{k-1}\geq \overline{n}_k$\,. Thus we have
        $$\sum\limits^k_{i=1}\overline{n}_i=n\,\,\,\mbox{ and  }\,\,\,\sum\limits^k_{i=1}\overline{m}_i=\overline{m}=p.$$

 From the above, it follows that
    $$p=\sum\limits^k_{i=1}\overline{m}_i\geq \sum\limits^k_{i=1}(\overline{n}_i-1)=n-k,\,\,\,\mbox{ that is, }\,\,k\geq n-p.$$

 \vspace*{3mm}

 Therefore there are at least $n-p$ Laplacian eigenvalues which are zero
 in $\overline{G}$\,, that is,
 \begin{equation}
 \overline{\mu}_i=0,\,\,i=p+1,\,p+2,\ldots,\,n.\label{1e0}
 \end{equation}

 Using the above, we get
 \begin{equation}
 \sum\limits^{n-1}_{i=1}\overline{\mu}_i=\sum\limits^{p}_{i=1}\overline{\mu}_i=2p.\label{1e1}
 \end{equation}

 \vspace*{3mm}

 Since $\overline{G}$ is disconnected, by Lemma \ref{m2}, we have
      $$\overline{\mu}_i\leq n-1,\,\,\,i=1,\,2,\ldots,\,n-1.$$

 Now we have
 \begin{eqnarray}
 Kf(G)&=&\sum\limits^{n-1}_{i=1}\frac{n}{\mu_i}\nonumber\\[2mm]
      &=&\sum\limits^{n-1}_{i=1}\frac{n}{n-\overline{\mu}_{n-i}}\,\,\,\mbox{ as $\mu_i=n-\overline{\mu}_{n-i}$}\nonumber\\[2mm]
      &=&n-1-p+\sum\limits^{p}_{i=1}\frac{n}{n-\overline{\mu}_i}\,\,\,\mbox{ by (\ref{1e0})}\nonumber\\[2mm]
      &\geq&n-1-p+\frac{p^2}{\sum\limits^{p}_{i=1}\displaystyle{\frac{n-\overline{\mu}_i}{n}}}\,\,\,\mbox{ by AM and HM inequality}\label{ps2}\\[2mm]
      &=&n-p-1+\frac{p}{1-2/n}\,\,\,\mbox{ as }\sum\limits^{p}_{i=1}\overline{\mu}_i=2p\nonumber\\[2mm]
      &=&n-1+\frac{2p}{n-2}\,.\nonumber
 \end{eqnarray}
 First part of the proof is done.

 \vspace*{3mm}

 Now suppose that the equality holds in (\ref{ps1}). Then all inequalities in the above argument must be equalities.
 From the equality in (\ref{ps2}), we get
    $$\frac{n}{n-\overline{\mu}_1}=\frac{n}{n-\overline{\mu}_2}=\cdots=\frac{n}{n-\overline{\mu}_p}\,,\,\,\,\mbox{ that is,
    }\,\,\overline{\mu}_1=\overline{\mu}_2=\cdots=\overline{\mu}_p\,.$$

 Using (\ref{1e1}), from the above, we get
      $$\overline{\mu}_1=\overline{\mu}_2=\cdots=\overline{\mu}_p=2.$$

 From the above, we conclude that each connected component $(n_i\geq 2)$ is isomorphic to
 $K_2$, otherwise, the largest Laplacian eigenvalue in $\overline{G}$ is $\overline{\mu}_1\geq 3$, a
 contradiction. Hence $\overline{G}\cong pK_2\cup (n-2p)K_1=pK_2\cup\overline{ K_{n-2p}}$\,,
 that is, $G\cong K_n-p\,K_2$\,.

 \vspace*{3mm}

 Conversely, let $G$ be isomorphic to the graph $K_n-p\,K_2$\,. Then the Laplacian spectrum of $G$ is
     $$S(G)=\{n^{(n-p-1)}\,,(n-2)^{(p)},\,0\}\,.$$
 Hence the equality holds in (\ref{ps1}).
 \end{proof}

 \begin{lemma} {\rm(\cite{Merris1994})} \label{dk7} Let $G$ be a simple graph on $n$ vertices which has at least one edge. Then
 \begin{equation}
 \mu_1\geq \Delta+1\,,\label{dlu8}
 \end{equation}
 where $\Delta$ is the maximum degree in $G$\,. Moreover, if $G$ is connected, then the
 equality holds in {\em(\ref{dlu8})} if and only if $\Delta=n-1$.
 \end{lemma}

 Let $a_1,\,a_2,\ldots,\,a_n$ be positive real numbers. We define $A_k$ to be the average of all products of $k$ of the $a_i$'s, that
 is,
 \begin{eqnarray}
 A_1&=&\frac{a_1+a_2+\cdots+a_n}{n}\nonumber\\
    &&\nonumber\\
 A_2&=&\frac{a_1a_2+a_1a_3+\cdots+a_1a_n+a_2a_3+\cdots+a_{n-1}a_n}{\frac{1}{2}\,n(n-1)}\nonumber\\
 \vdots\nonumber\\
 A_{n-1}&=&\frac{a_2\ldots a_{n-1}\,a_n+a_1a_3\ldots a_{n-1}\,a_n+\cdots+a_1a_2\ldots a_{n-2}\,a_n+a_1a_2\ldots a_{n-1}}{n}\nonumber\\
 A_n&=&a_1a_2\ldots a_n\,.\nonumber
 \end{eqnarray}

 Hence the AM is simply $A_1$ and the GM is $A_n^{1/n}$\,. The
following result generalize this:

 \begin{lemma} {\rm (Maclaurin's Symmetric Mean Inequality \cite{BW})} \label{1m2}
 For positive real numbers $a_1,~a_2,\ldots,~a_n$,
    $$A_1\geq A_2^{1/2}\geq A_3^{1/3}\geq \ldots \geq A_{n-1}^{1/(n-1)}\geq A_n^{1/n}.$$
 Equality holds if and only if $a_1=a_2=\cdots=a_n$.
 \end{lemma}

 \begin{theorem}\label{Weak} For any integer $2\leq p\leq \lfloor\frac{n}{2}\rfloor$ and any graph $G$ obtained by deleting $p$ edges from $K_n$, we have
 \begin{eqnarray}
 Kf(G)\leq n-1-p+\frac{n}{n-p-1}+\frac{(p-1)\,\delta\,n^{n-p-1}\,(n-1)^{p-2}}{t(G)}\,,\label{e1ps1}
 \end{eqnarray}
 where $t(G)$ is the number of spanning trees in $G$ and $\delta$ is the minimum degree in $G$. Moreover, the equality holds in $(\ref{e1ps1})$
 if and only if $G\cong K_n-K_{1,\,p}$.
 \end{theorem}

 \begin{proof} For the sake of consistency,  $\overline{\mu}_i$ with $i=1,\,2,\ldots,\,n$, $\overline{m}$, $\overline{m}_i$ and $\overline{n}_i$ are similarly defined  as that
 in the proof of Theorem \ref{NEWTH1}. Then we claim that $\overline{G}$ has exactly $n-p$ components of order $n$ and with $p$ edges. It follows that
 \begin{equation}
 \overline{\mu}_i=0,\,\,i=p+1,\,p+2,\ldots,\,n,~\mbox{ that is, }~\mu_i=n,~i=1,\,2,\ldots,\,n-p-1.\label{1ew11}
 \end{equation}

 Moreover, we have
 \begin{equation}
 \sum\limits^p_{i=1}\,\overline{\mu}_i=2p. \label{dase1}
 \end{equation}

 Now we assume that $\overline{G}=\bigcup\limits_{i=1}^{n-p}H_{i}$ and $\overline{\Delta}$ is the maximum degree in $\overline{G}$. Then, by Lemmas \ref{m2} and \ref{dk7}, we have
 \begin{eqnarray}
 \overline{\Delta}+1\leq \overline{\mu}_1=\max\limits_{1\leq i\leq n-p}\mu_1(H_i)\leq p+1.\label{1ew1}
 \end{eqnarray}

 Putting $n=p-1$ and $a_i=n-\overline{\mu}_{i+1}$, $i=1,\,2,\ldots,\,p-1$ in Lemma \ref{1m2}, we get $A_1\geq A_{p-2}^{1/(p-2)}$\,, that is,
 \begin{equation}
 \frac{\sum\limits^p_{i=2}\,(n-\overline{\mu}_i)}{p-1}\geq \left[\frac{\prod^p_{i=2}\,(n-\overline{\mu}_i)\,\sum\limits^{p}_{i=2}
 \frac{1}{n-\overline{\mu}_i}}{p-1}\right]^{1/(p-2)}\,.\label{kcd1}
 \end{equation}

 It is well known that
     $$t(G)=\frac{1}{n}\,\prod^{n-1}_{i=1}\,\mu_i\,.$$

 \vspace*{3mm}

 Since $n-\mu_{n-1}=\overline{\mu}_1\geq \overline{\Delta}+1$ and $n-\overline{\Delta}-1=\delta$, we have
  $$\prod^p_{i=2}(n-\overline{\mu}_i)=\prod^p_{i=2}\,\mu_{n-i}=\frac{\prod^{n-1}_{i=1}\,\mu_{i}}{\prod^{n-p-1}_{i=1}\,\mu_{i}\cdot \mu_{n-1}}\geq
  \frac{n\,t(G)}{n^{n-p-1}\,\delta}$$
 and
   $$\frac{\sum\limits^p_{i=2}\,(n-\overline{\mu}_i)}{p-1}=\frac{n(p-1)-(2p-\overline{\mu}_1)}{p-1}\leq
   n-1~~\mbox{ as }\overline{\mu}_1\leq p+1.$$

 \vspace*{3mm}

 Using the above result in (\ref{kcd1}), we get
 \begin{eqnarray}
 \sum\limits^{p}_{i=2} \frac{1}{n-\overline{\mu}_i}&\leq& \frac{(p-1)}{\prod^p_{i=2}\,(n-\overline{\mu}_i)}\,(n-1)^{p-2}\nonumber\\[3mm]
     &\leq&\frac{(p-1)\,(n-1)^{p-2}\,\delta\,n^{n-p-1}}{n\,t(G)}\,.\label{kcd2}
 \end{eqnarray}

 Therefore, we have
 \begin{eqnarray}
 Kf(G)&=&\sum\limits^{n-1}_{i=1}\frac{n}{\mu_i}\nonumber\\[2mm]
      &=&\sum\limits^{n-1}_{i=1}\frac{n}{n-\overline{\mu}_{n-i}}\,\,\,\mbox{ as $\mu_i=n-\overline{\mu}_{n-i}$}\nonumber\\[2mm]
      &=&n-1-p+\sum\limits^{p}_{i=1}\frac{n}{n-\overline{\mu}_i}\,\,\,\mbox{ by (\ref{1ew11})}\nonumber\\
      &\leq&n-1-p+\frac{n}{n-p-1}+\sum\limits^{p}_{i=2}\frac{n}{n-\overline{\mu}_i}~~~\mbox{ by (\ref{1ew1})}\,.\nonumber
 \end{eqnarray}

 Using (\ref{kcd2}) in the above, we get the required result in (\ref{e1ps1}). First part of the proof is done.

 \vspace*{3mm}

 Now suppose that the equality holds in (\ref{e1ps1}). Then all inequalities in the above argument must be equalities.
 From the equality in (\ref{kcd1}), we get $\overline{\mu}_2=\overline{\mu}_3=\cdots=\overline{\mu}_p$, by Lemma \ref{1m2}.

 \vspace*{3mm}

 From the equality in (\ref{kcd2}), we get $\overline{\mu}_1=p+1$. Using (\ref{dase1}) with the above results, we get $\overline{\mu}_2=\overline{\mu}_3=\cdots=\overline{\mu}_p=1$.
 Thus we must have $\overline{G}$ is tree $K_{1,\,p}$ and all the remaining $n-p-1$ components are trivially $K_1$'s. Equivalently, we deduce that $G=K_n-K_{1,\,p}$.

 \vspace*{3mm}

 Conversely, let $G\cong K_n-K_{1,\,p}$\,. Then we have $\mu_1=\mu_2=\cdots=\mu_{n-p-1}=n$, $\mu_{n-p}=\mu_{n-p+1}=\cdots=\mu_{n-2}=n-1$ and
 $\mu_{n-1}=n-p-1$. Also we have $t(G)=(n-p-1)\,n^{n-p-2}\,(n-1)^{p-1}$ and $\delta=n-p-1$. Now,
 \begin{eqnarray}
 &&n-1-p+\frac{n}{n-p-1}+\frac{(p-1)\,\delta\,n^{n-p-1}\,(n-1)^{p-2}}{t(G)}\nonumber\\[3mm]
 &&~~~~~~~~~~~~~~~~~~~~~~~~~~~~~~~~~~~~~~~~~=n-p-1+\frac{n}{n-1}\,(p-1)+\frac{n}{n-p-1}\nonumber\\[2mm]
 &&~~~~~~~~~~~~~~~~~~~~~~~~~~~~~~~~~~~~~~~~~=Kf(K_n-K_{1,p})\,.\nonumber
 \end{eqnarray}
 This completes the proof.
 \end{proof}

 The following lemma was implicitly proved in \cite{Kel1974}.
\begin{lemma}{\em(\cite{Kel1974})} \label{NADD1} Let $G$ be a connected graph obtained by deleting $p\leq n-1$ edges from the complete graph $K_n$. Then we have \begin{eqnarray}
 t(G)\geq n^{n-p-2}(n-1)^{p-1}(n-p-1)\,,\label{eADD1}
 \end{eqnarray}
 with equality holding if and only if $G\cong K_n-K_{1,p}$. \end{lemma}
 Combining Lemma \ref{NADD1} and Theorem \ref{Weak}, we can easily deduce the following corollary.

 \begin{corollary} \label{NCO1} For any integer $2\leq p\leq \lfloor\frac{n}{2}\rfloor$ and any graph $G$ obtained by deleting $p$ edges from $K_n$, we have
 \begin{eqnarray}
 Kf(G)\leq n-1-p+\frac{n}{n-p-1}+\frac{n\,(p-1)\,\delta}{(n-1)(n-p-1)}\,,\label{e1ps22}
 \end{eqnarray}
 where $\delta$ is the minimum degree in $G$. Moreover, the equality holds in $(\ref{e1ps1})$
 if and only if $G\cong K_n-K_{1,\,p}$. \end{corollary}
\medskip
\subsection{The ordering of connected graphs with larger Kirchoff indices}
\ \indent In this subsection we will determine the graphs from ${\cal{G}}(n)$ ($n>27$) with first to ninth largest Kirchhoff indices.  Considering Lemma \ref{NEW1} $(1)$ and Corollary \ref{CO1}, we find that the path $P_n$ has the largest Kirchhoff index among all graphs from ${\cal{G}}(n)$.
Before stating our main result, we first prove a lemma below.
\begin{lemma} \label{LEM3.1} For any connected graph $G$ of order $n$ and with $m>n+1$ edges, there exists a connected graph $G_1$ of order $n$ and with $n+1$ edges such that $Kf(G_1)>Kf(G)$.
\end{lemma}

\begin{proof} For any connected graph $G$ of order $n$ with $m>n+1$ edges, choosing and deleting one non-cut edge from $G$, we can get a connected graph $G'$ of order $n$
with $m-1$ edges and $Kf(G')>Kf(G)$ by Lemma \ref{NEW1} $(1)$.
Repeating the above process by $m-n-1$ times, we can obtain a
connected graph $G_1$ of order $n$ with $n+1$ edges and
$Kf(G_1)>Kf(G)$, completing the proof of this lemma.
\end{proof}

Now we denote by $Q_n^k$ (see Figure \ref{G3}
for the case when $k=3$) the graph obtained by attaching a pendent
edge to the unique neighbor of the pendent vertex in
$P_{n-1}^k$. Let $R_n^3$ be a graph, shown in Figure
\ref{G3}, which is obtained by attaching a pendent edge to the
vertex with the distance $2$ from the pendent vertex in
$P_{n-1}^3$. A graph $CQ_n^{3}$ is obtained by attaching a pendent edge to a vertex of $C_3$ in $Q_{n-1}^3$ with degree $2$.  Let $C_3(k_1,k_2)$ be a graph obtained attaching a path of length $k_1$ to one vertex of $C_3$ and a path of length $k_2$ to another vertex in $C_3$. Denote by $C_3(k_1,k_2,k_3)$ a graph obtained by attaching three paths of lengths $k_1$, $k_2$ and $k_3$, respectively, to three vertices of $C_3$. In the following we define two sets of graphs:
\begin{eqnarray*}{\cal{H}}(n)=\Big\{P_n^3,Q_n^3,R_n^3,C_3(1,n-4),C_3(2,n-5),CQ_n^{3} \Big\}, \end{eqnarray*}  \begin{eqnarray*}{\cal{T}}^0(n)=\Big\{P_n,T_n(n-3,1^2),T_n(n-4,2,1),T_n(1^2;1^2),T_n(n-5,3,1),T_n(1^2;2,1)\Big\}. \end{eqnarray*}
\begin{figure}[ht!]
\begin{center}
\begin{tikzpicture}[scale=1,style=thick]
\def\vr{3pt}

\draw (-5.0,0.5) -- (-5.0,-0.5) -- (-4.2,0) -- (-5.0,0.5);
\draw (-4.2,0) -- (-3.4,0);
\draw (-1.8,0) -- (-1.0,0) -- (-0.2,-0.5);
\draw (-1.0,0) -- (-0.2,0.5);
\draw (6.8,0) -- (6.0,0) -- (5.2,0) -- (5.2,0.6);
\draw (5.2,0) -- (4.4,0);
\draw (1.5,0.5) -- (1.5,-0.5) -- (2.3,0) -- (1.5,0.5);
\draw (2.3,0) -- (3.0,0);

\draw (-5.0,0.5)  [fill=black] circle (\vr);
\draw (-5.0,-0.5)  [fill=black] circle (\vr);
\draw (-4.2,0)  [fill=black] circle (\vr);
\draw (-3.4,0)  [fill=black] circle (\vr);
\draw (-1.8,0)  [fill=black] circle (\vr);
\draw (-1.0,0)  [fill=black] circle (\vr);
\draw (-0.2,0.5)  [fill=black] circle (\vr);
\draw (-0.2,-0.5)  [fill=black] circle (\vr);

\draw (1.5,0.5)  [fill=black] circle (\vr);
\draw (1.5,-0.5)  [fill=black] circle (\vr);
\draw (2.3,0)  [fill=black] circle (\vr);

\draw (3.0,0)  [fill=black] circle (\vr);
\draw (4.4,0)  [fill=black] circle (\vr);
\draw (5.2,0.6)  [fill=black] circle (\vr);
\draw (5.2,0)  [fill=black] circle (\vr);
\draw (6.0,0)  [fill=black] circle (\vr);

\draw (6.8,0) [fill=black] circle (\vr);

\draw (-2.6,0) node {$\cdots\cdots$};
\draw (3.6,0) node {$\cdots\cdots$};
\draw (-2.8,-0.8) node {$Q_n^3$};
\draw (4.2,-0.8) node {$R_n^3$};

\end{tikzpicture}
\end{center}
\caption{ The graphs $Q_n^3$ and $R_n^3$}
\label{G3}
\end{figure}
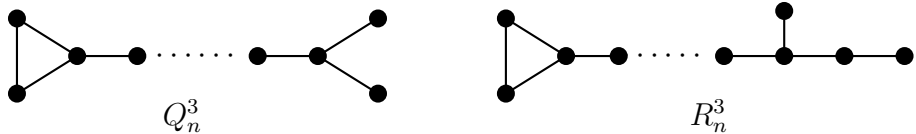

 It is not difficult to verify that any spanning tree of the graphs $C_3(1,1)$ and $C_3(1,1,1)$ must be in the set ${\cal{T}}^0(n)$.
\begin{lemma}\label{LEM3.2} Let $G$ be a connected graph of order $n$ $(n\geq 10)$ with $n$ edges and maximum degree $\Delta\geq 3$,
cycle length $k>4$. Then $G$ has a spanning tree $T$ with $T\notin {\cal{T}}^0(n)$.
\end{lemma}

\begin{proof} Assume that $G$ contains a cycle  $C_k$ as a subgraph. According to the value of $\Delta$, we divide into the following two cases.

 \textbf{Case 1.} $\Delta\geq 4$.

In this case, we choose $G-e$  where $e$ is on the cycle in $G$ but not incident with the vertex of degree $\Delta$ in it. Then $G-e$ is a spanning tree of $G$ with maximum degree $\Delta\geq 4$. Then $G-e\notin {\cal{T}}^0(n)$,  since any tree in ${\cal{T}}^0(n)$ has maximum degree $3$.

\textbf{ Case 2.} $\Delta=3$.

Assume that $v$ is a vertex in $C_k$ of degree $3$ in $G$. Note that $k\geq 5$ from the condition in this lemma.
Now we choose an edge $e=v_1v_2$ on the cycle $C_k$ in
$G$ such that $v_1$ and $v_2$ are all in the distance as large as
possible from the vertex $v$. Since $k\geq 5$, we have $d_G(v,v_1)\geq 2$ and
$d_G(v,v_2)\geq 2$. Then $G-e$ is a spanning tree of
$G$ with $G-e\notin {\cal{T}}^0(n)$, since neither of neighbors of $v$
are pendent vertices. \end{proof}
\begin{lemma}\label{LEM3.3} Let $G\notin {\cal{H}}(n)$ be a connected graph of order $n$ $(n\geq 8)$ with $n$ edges and maximum degree $\Delta\geq3$,
cycle length $k=3$. Then $G$ has a spanning tree $T$ with $T\notin {\cal{T}}^0(n)$.
\end{lemma}
\begin{proof} For the case $\Delta>3$, from a similar reasoning as that in Case 1 in the proof of Lemma \ref{LEM3.2}, our result follows immediately. Therefore  it suffices to consider the case $\Delta=3$. Assume that $C_3=v_1v_2v_3v_1$ in $G$. Next we deal with the following three cases.

\textbf{Case 1.} There is only one vertex, say $v_1$, of $C_3$ in $G$ with degree $3$.

In this case,  we choose the  edge $e=v_2v_3$ in $C_3$. Then $G-e$ is a spanning tree of $G$, in which the vertex $v_1$ is still of degree $3$. Thus we have $G-e\ncong P_n$. If $G-e\cong T_n(n-3,1^2)$, then the super graph $G$ obtained by inserting the edge $e$ into $T_n(n-3,1^2)$ is just $P_n^3$, contradicting the fact that $G\notin {\cal{H}}(n)$. Therefore $G-e\ncong T_n(n-3,1^2)$. By a similar reasoning, we can conclude that $G-e\ncong T_n(1^2;1^2)$ for the edge $e\in E(G)$ defined as above from the condition that $G\ncong Q_n^3$. Moreover, if $G-e\cong T_n(1^2;2,1)$ for the edge $e$ in the triangle in $G$ and not incident with the vertex $v$ in it, then we claim that $G\cong R_n^3$.  This is impossible because of the fact that $G\notin {\cal{H}}(n)$. Therefore, we have $G-e\notin {\cal{T}}^0(n)$.

\textbf{Case 2. } There are exactly two vertices, say $v_1$ and $v_2$,  of $C_3$ in $G$ with degree $3$.

In this case, without loss of generality, we assume that the eccentricity of $v_1$ is not more than that of $v_2$ in $G$. Let $e=v_2v_3$. Then $G-e$ is a spanning tree in $G$. Since $G\ncong C_3(1,n-4)$, we deduce that $G-e\ncong T_n(n-4,2,1)$. Similarly, we have $G-e\ncong T_n(n-5,3,1)$ from the condition $G\ncong C_3(2,n-5)$. Moreover, $G-e\ncong T_n(1^2;2,1)$, since $G\ncong CQ_n^3$.  Note that, in $G-e$, there are at least two pendent vertices at the distance $d\geq 2$ to $v_1$ with degree $3$. Therefore we have $G-e\notin  {\cal{T}}^0(n)$ as desired.

\textbf{Case 3.} All the vertices of $C_3$ in $G$ are of degree $3$.

 Assume that $v_1$ has the smallest eccentricity among all the vertices of $C_3$ in $G$. Let $e=v_2v_3$. Then $G-e$ is a spanning tree of $G$ such that $v_1$ is of degree $3$ in it. Moreover, $G-e\notin  {\cal{T}}^0(n)$, since there are at least three pendent vertices at the distance at least $2$ to $v_1$ in $G-e$. This completes the proof for this case, ending the proof of this lemma.
\end{proof}
\begin{theorem}\label{TH1} Let $n>27$. Then we have
\begin{eqnarray*}Kf(P_n)>Kf(T_n(n-3,1^2))>Kf(P_n^3)>Kf(T_n(n-4,2,1))>Kf(T_n(1^2;1^2))>Kf(Q_n^3)\\
>Kf(T_n(n-5,3,1))>Kf(T_n(n-4,1^3))=Kf(T_n(1^2;2,1))>Kf(C_{3,3}^{n-5}).\end{eqnarray*}\end{theorem}
\begin{proof}In view of Corollary \ref{CO1} and Lemma \ref{LEM6}, considering Lemma \ref{NEW1} $(2)$, we claim that the remaining is only to prove the following inequalities:
 \begin{equation}Kf(P_n^3)>Kf(T_n(n-4,2,1)),~\label{e3}  \end{equation}
 \begin{equation}Kf(Q_n^3)>Kf(T_n(n-5,3,1)),~\label{e4}\end{equation}
 \begin{equation}Kf(T_n(1^2;2,1))>Kf(C_{3,3}^{n-5}).~\label{e5} \end{equation}
 From Lemma \ref{NEW4} and the results in \cite{LLL2010}, we have
 \begin{eqnarray}Kf(T_n(n-4,2,1))=\displaystyle{{n+1\choose 3}}-2n+8=\displaystyle{\frac{n^3-13n+48}{6}},\nonumber\end{eqnarray} \begin{eqnarray}Kf(T_n(n-5,3,1))=\displaystyle{{n+1\choose 3}}-3n+15=\displaystyle{\frac{n^3-19n+90}{6}},\nonumber\end{eqnarray} \begin{eqnarray}Kf(T_n(1^2;2,1))=\displaystyle{{n+1\choose 3}}-3n+11=\displaystyle{\frac{n^3-19n+66}{6}}.\nonumber\end{eqnarray}

 By Lemmas \ref{LEM6} and \ref{LEM7}, we arrive at the following results: $$Kf(P_n^3)=\displaystyle{\frac{n^3-11n+18}{6}}, ~~~ Kf(C_{3,3}^{n-5})=\displaystyle{\frac{n^3-21n+36}{6}}.$$

 Some straightforward calculations show the validity of inequalities (\ref{e3}) and (\ref{e5}) for $n>27$.

Setting $T'=T_{n-2}(n-5,1^2)$ and applying Lemma \ref{LEM8} to the vertex of degree $3$ in $C_3$ of $Q_n^3$, we have
\begin{eqnarray*}
Kf(Q_n^3)&=&Kf(C_3)+Kf(T')+2Kf_x(T')+(n-2)Kf_x(C_3)\\[3mm]
&=&2+\frac{(n-2)^3-7(n-2)+18}{6}\\[3mm]
&&+2\Big[1+2+3+\cdots\cdots+(n-5)+2(n-4)\Big]+\frac{4}{3}(n-3)\\[3mm]
&=& \frac{n^3-17n+36}{6}.
\end{eqnarray*}
It can be easily checked that
$\displaystyle{\frac{n^3-17n+36}{6}}>\frac{n^3-19n+90}{6}$ when
$n>27$, i.e., the inequality (\ref{e4}) holds if $n>27$. This
completes the proof of this theorem.
\end{proof}

Now we define a new set of graphs as follows:
\begin{eqnarray*}{\cal{G}}^0(n)={\cal{T}}^0(n)\bigcup\left\{T_n(n-4,1^3),P_n^3,Q_n^3,C_{3,3}^{n-5}\right\}.
\end{eqnarray*}

In the following theorem we order the graphs from ${\cal{G}}(n)$ with first to tenth largest Kirchhoff indices.
\begin{theorem}\label{TH2} Let $G$ be any graph from ${\cal{G}}(n)\setminus{\cal{G}}^0(n)$ with $n>27$. Then we have
\begin{eqnarray*}Kf(P_n)>Kf(T_n(n-3,1^2))>Kf(P_n^3)>Kf(T_n(n-4,2,1))>Kf(T_n(1^2;1^2))
>Kf(Q_n^3)\\>Kf(T_n(n-5,3,1))>Kf(T_n(n-4,1^3))=Kf(T_n(1^2;2,1))
>Kf(C_{3,3}^{n-5})>Kf(G).\end{eqnarray*}\end{theorem}
\begin{proof} By Theorem \ref{TH1}, it suffices to prove that $Kf(G)<Kf(C_{3,3}^{n-5})$ for any graph $G\in {\cal{G}}(n)\setminus{\cal{G}}^0(n)$ with $n>27$.

 If $G\in {\cal{G}}(n)\setminus{\cal{G}}^0(n)$ has $m>n+1$ edges,  by Lemma \ref{LEM3.1}, we conclude that there exists a connected graph $G_1$ of order $n$ and with $n+1$ edges
 such that $Kf(G)<Kf(G_1)$. By Lemma \ref{LEM7}, we have $Kf(G)<Kf(G_1)\leq Kf(C_{3,3}^{n-5})$. Clearly, for any connected graph $G$ of order $n$ and with
 $n+1$ edges,  $Kf(G)<Kf(C_{3,3}^{n-5})$ from Lemma \ref{LEM7}, again.

 Now we only need to consider the connected graphs of order $n$ and with $m\leq n$ edges. In the case when $m=n-1$ with $n>27$, for any graph $G\notin {\cal{T}}^0(n)\bigcup\{T_n(n-4,1^3)\}$
 of order $n$ and with $n-1$ edges, i.e., $G$ is a tree, by Corollary \ref{CO1} and Lemma \ref{LEM8}, we have
\begin{eqnarray*}
 Kf(G)&\leq&Kf(T_n(n-6,4,1))\\[3mm]
 &=&{n+1\choose 3}-4n+24\\[3mm]
 &=&\frac{n^{3}-25n+144}{6}\\[3mm]
 &<&\frac{n^{3}-21n+36}{6}\\[3mm]
 &=&Kf(C_{3,3}^{n-5}).
\end{eqnarray*}

Now we focus on the case when $m=n$. Combining Lemma
\ref{LEM3.2} and Corollaries \ref{CO0} and \ref{CO1}, we find that,
when $n>27$,  for any connected graph $G$ of
order $n$ and with $n$ edges, maximum degree $\Delta\geq 3$ and
cycle length $k>4$, we have $Kf(G)\leq
Kf(T_n(n-6,4,1))<Kf(C_{3,3}^{n-5})$. By Lemma \ref{LEA9}, we have
$Kf(G)\leq Kf(P_n^4)$ for any connected graph $G$ of order $n$
and with $n$ edges and cycle length $4$. From Lemma \ref{LEM3.3}, Corollaries \ref{CO0} and \ref{CO1}, we have $Kf(G)\leq
Kf(T_n(n-6,4,1))<Kf(C_{3,3}^{n-5})$ for any graph $G\notin {\cal{H}}(n)$ of order $n$ with $n$ edges, cycle length $3$ and maximum degree $\Delta$. Thus the remaining for this case is to show that $Kf(G)<Kf(C_{3,3}^{n-5})$ for any graph $G$ from the set $\{R_n^3,P_n^4,C_n,C_3(1,n-4),C_3(2,n-5),CQ_n^3\}$. From Corollary \ref{CAD1}, $Kf(CQ_n^3)<Kf(C_3(1,n-4))$. Note that $Kf(P_n^3)=\displaystyle{\frac{n^3-11n+18}{6}}$ and $Kf(P_n)=\displaystyle{\frac{n^3-n}{6}}$ (\cite{LLL2010}). Applying Lemma \ref{LEM8} to the vertices in $C_3$ of $C_3(1,n-4),C_3(2,n-5)$, respectively, with degree $3$ and a smaller eccentricity, we have $$Kf(C_3(1,n-4))=\frac{n^3-27n+82}{6},~~~Kf(C_3(2,n-5))=\frac{n^3-25n+88}{6}, $$
 both of them is less than $\displaystyle{\frac{n^{3}-21n+36}{6}}=Kf(C_{3,3}^{n-5})$. Moreover, we have $Kf(CQ_n^3)<Kf(C_{3,3}^{n-5})$. By the formula
 $$Kf(P_n^l)=\displaystyle{\frac{n^{3}-2n}{6}}+\frac{(1+2n)l}{4}+\frac{l^{3}}{4}-\frac{(3+2n)l^{2}}{6}$$
 in \cite{Yang2008}, we can get
   $$Kf(P_n^4)=\displaystyle{\frac{n^{3}-22n+54}{6}}<\frac{n^{3}-21n+36}{6}=Kf(C_{3,3}^{n-5})~~\mbox{ when }~n>27.$$
 Also from \cite{Yang2008}, we have $Kf(C_n)=\displaystyle{\frac{n^3-n}{12}}$.
 Therefore it follows that
\begin{eqnarray*}
Kf(C_n)&=&\frac{n^3-n}{12}\\[3mm]
&<&\frac{n^{3}-21n+36}{6}\\[3mm]
 &=&Kf(C_{3,3}^{n-5})~~as~~ n^3-41n+72>0~~\mbox{ when }~~ n>27.
\end{eqnarray*}
 Finally, setting $T''=T_{n-2}(n-6,2,1)$, by the application of Lemma \ref{LEM8} to the vertex, say $x$, of degree $3$ on the triangle $C_3$ of $R_n^3$, we have
\begin{eqnarray*}Kf(R_n^3)&=&Kf(C_3)+Kf(T'')+(n-3)Kf_x(C_3)+2Kf_x(T'')\\[3mm]
&=&2+\frac{(n-2)^2-13(n-2)+48}{6}+\frac{4}{3}(n-3)\\[3mm]
&&+2\Big[1+2+\cdots+(n-5)+(n-4)+(n-5)\Big]\\[3mm]
&=&\frac{n^3-23n+66}{6}.
\end{eqnarray*}
Obviously, we conclude that
   $$Kf(R_n^3)=\displaystyle{\frac{n^3-23n+66}{6}}<\frac{n^{3}-21n+36}{6}=Kf(C_{3,3}^{n-5})~~\mbox{ if }~n>27.$$
Thus we complete the proof of this theorem.
\end{proof}

\end{document}